\documentclass[12pt]{amsart}
\usepackage{amscd,amsmath,amsthm,amssymb}
\usepackage{tikz}
\tikzstyle{punkt}=[circle, fill=black, minimum size=1mm,inner sep=0pt, draw]

\usepackage{amsfonts,amssymb,amscd,amsmath,enumerate,verbatim}

\usetikzlibrary{arrows}
\input xy
\xyoption{all}
\def\frk{\mathfrak}               

\def\Phi{{\frk N}}
\def\opn#1#2{\def#1{\operatorname{#2}}} 
\opn\chara{char} \opn\length{\ell} \opn\pd{pd} \opn\rk{rk}
\opn\projdim{proj\,dim} \opn\injdim{inj\,dim} \opn\rank{rank}
\opn\depth{depth} \opn\grade{grade} \opn\height{height}
\opn\embdim{emb\,dim} \opn\codim{codim}

\opn\Tr{Tr} \opn\bigrank{big\,rank}
\opn\superheight{superheight}\opn\lcm{lcm}
\opn\trdeg{tr\,deg}
\opn\reg{reg} \opn\lreg{lreg} \opn\ini{in} \opn\lpd{lpd}
\opn\size{size}\opn{\mult}{mult}
\opn\div{div} \opn\Div{Div} \opn\cl{cl} \opn\Cl{Cl}
\opn\Spec{Spec} \opn\Supp{Supp} \opn\supp{supp} \opn\Sing{Sing}
\opn\Ass{Ass} \opn\Min{Min}
\opn\Ann{Ann} \opn\Rad{Rad} \opn\Soc{Soc}
\opn\Syz{Syz} \opn\Im{Im} \opn\Ker{Ker} \opn\Coker{Coker}
\opn\Am{Am} \opn\Hom{Hom} \opn\Tor{Tor} \opn\Ext{Ext}
\opn\End{End} \opn\Aut{Aut} \opn\id{id} \opn\ini{in}

\opn\nat{nat}
\opn\pff{pf}
\opn\Pf{Pf} \opn\GL{GL} \opn\SL{SL} \opn\mod{mod} \opn\ord{ord}
\opn\Gin{Gin}
\opn\Hilb{Hilb}\opn\adeg{adeg}\opn\std{std}\opn\ip{infpt}
\opn\Pol{Pol}
\opn\sat{sat}
\opn\Var{Var}
\opn\Gen{Gen}

\opn\aff{aff} \opn\con{conv} \opn\relint{relint} \opn\st{st}
\opn\lk{lk} \opn\cn{cn} \opn\core{core} \opn\vol{vol}
\opn\link{link} \opn\star{star}
\opn\gr{gr}

\def\pot#1#2{#1[\kern-0.28ex[#2]\kern-0.28ex]}

\opn\dirlim{\underrightarrow{\lim}}
\opn\inivlim{\underleftarrow{\lim}}
\def\Implies{\ifmmode\Longrightarrow \else
        \unskip${}\Longrightarrow{}$\ignorespaces\fi}
\def\implies{\ifmmode\Rightarrow \else
        \unskip${}\Rightarrow{}$\ignorespaces\fi}
\def\iff{\ifmmode\Longleftrightarrow \else
        \unskip${}\Longleftrightarrow{}$\ignorespaces\fi}

\let\:=\colon
\newtheorem{Theorem}{Theorem}[section]
\newtheorem{Lemma}[Theorem]{Lemma}
\newtheorem{Corollary}[Theorem]{Corollary}
\newtheorem{Proposition}[Theorem]{Proposition}
\newtheorem{Remark}[Theorem]{Remark}

\newtheorem{Example}[Theorem]{Example}

\newtheorem{Question}[Theorem]{Question}
\let\epsilon\varepsilon
\let\phi=\varphi
\let\kappa=\varkappa
\def\qed{\ifhmode\textqed\fi
      \ifmmode\ifinner\quad\qedsymbol\else\dispqed\fi\fi}
\def\textqed{\unskip\nobreak\penalty50
       \hskip2em\hbox{}\nobreak\hfil\qedsymbol
       \parfillskip=0pt \finalhyphendemerits=0}
\def\dispqed{\rlap{\qquad\qedsymbol}}

\opn\dist{dist}
\def\pnt{{\raise0.5mm\hbox{\large\bf.}}}

\opn\Lex{Lex}
\opn\diam{diam}

\begin{document}
	
\title{Binomial edge ideals of Cameron--Walker graphs}
\author[T.~Hibi]{Takayuki Hibi}
\address[Takayuki Hibi]
{Department of Pure and Applied Mathematics, 
	Graduate School of Information Science and Technology, 
	Osaka University, 
	Suita, Osaka 565-0871, Japan}
\email{hibi@math.sci.osaka-u.ac.jp}
\author[S.~Saeedi~Madani]{Sara Saeedi Madani}
\address[Sara Saeedi Madani]
{Department of Mathematics and Computer Science, Amirkabir University of Technology, Tehran, Iran, and School of Mathematics, Institute for Research in Fundamental Sciences, Tehran, Iran} 
\email{sarasaeedi@aut.ac.ir, sarasaeedim@gmail.com}

\begin{abstract}
	Let $G$ be a Cameron--Walker graph on $n$ vertices and $J_G$ the binomial edge ideal of $G$.  Let $S$ denote the polynomial ring in $2n$ variables over a field.  It is shown that the following conditions are equivalent: (i) $S/J_G$ is Cohen--Macaulay; (ii) $J_G$ is unmixed; (iii) $\dim (S/J_G) = n+1$; (iv) (a) $n = 3$ and $G$ is a path of length $2$ or (b) $n = 5$ and $G$ is a path of length $4$ or (c) $n=5$ and $G$ is obtained by attaching a path of length $2$ to a triangle. Moreover, the depth of $S/J_G$ is computed for a class of Cameron--Walker graphs, called minimal dense Cameron--Walker graphs. As an application, it is proved that finite graphs $G$ with $\depth(S/J_G)=6$ can have any number of vertices~$n\geq 6$. Finally, it is shown that given integers $t,n$ with $6\leq t\leq n+1$, there exists a finite connected graph $G$ with $\depth (S/J_G)=t$. 
\end{abstract}

\maketitle

\section*{Introduction}
Let $G$ be a finite graph on the vertex set $[n]=\{1,\ldots,n\}$ and the edge set $E(G)$ with no loop, no multiple edge and no isolated vertex.  Let $S=K[x_1, \ldots, x_n, y_1, \ldots, y_n]$ denote the polynomial ring in $2n$ variables over a field $K$.  The {\em binomial edge ideal} of $G$, introduced by \cite{HHHKR} and \cite{O} independently, is the ideal $J_G$ of $S$ generated by the binomials $x_iy_j - x_jy_i$ with $i<j$ and 
$\{i,j\} \in E(G)$.  

We are interested in the binomial edge ideal of a Cameron--Walker graph \cite{HHKO}.  A {\em Cameron--Walker} graph is a finite graph consisting of a connected bipartite graph with vertex partition $U \cup V$ such that there is at least one leaf edge attached to each vertex $i \in U$ and that there may be possibly some pendant triangles attached to each vertex $j \in V$ (\cite[Fig. 1]{HHKO}). 

Let $G$ be a Cameron--Walker graph. In the present paper, we study unmixedness of $J_G$ and Cohen--Macaulayness of $S/J_G$. We also discuss the Krull dimension as well as the depth of $S/J_G$. More precisely, the structure of the paper is as follows. In Section~\ref{CM section}, given a Cameron--Walker graph on $n$ vertices, it is shown that the following conditions are equivalent: (i) $S/J_G$ is Cohen--Macaulay; (ii) $J_G$ is unmixed; (iii) $\dim (S/J_G) = n+1$; (iv) Either (a) $n = 3$ and $G$ is a path of length $2$, or (b) $n = 5$ and $G$ is a path of length $4$, or (c) $n=5$ and $G$ is obtained by attaching a path of length $2$ to a triangle. In Section~\ref{depth section}, we introduce the class of minimal dense Cameron--Walker graphs and provide an explicit formula for $\depth(S/J_G)$ when $G$ belongs to this class. Indeed,  for such a graph $G$ it is shown that $\depth(S/J_G)=n-|V|+2$.  As an application of this result, it is also shown that for any $n\geq 6$, there exists a minimal dense Cameron--Walker graph $G$ with $n$ vertices and $\depth(S/J_G)=6$. Note that~$6$ is the smallest value for depth of binomial edge ideals which has not been combinatorially characterized yet. Finally, in Section~\ref{possible depth section}, we discuss all the possible values for depth of a binomial edge ideal. Indeed, we show that for any integers $t,n$ with $6\leq t\leq n$, there exists a finite connected graph with $n$ vertices such that $\depth(S/J_G)=t$.

\section{Cameron--Walker graphs with Cohen--Macaulay binomial edge ideals}\label{CM section}

In this section, we consider unmixed and Cohen--Macaulay properties as well as the Krull dimension of $S/J_G$ where $G$ is a Cameron--Walker graph. First we recall the minimal prime ideals of binomial edge ideals in general.  

Let $G$ be a finite connected graph on $[n]$. Recall that an \emph{induced subgraph} of $G$ on $W\subseteq [n]$, denoted by $G|_{W}$, is a graph whose vertex set is $W$ and its edge set consists of those edges of $G$ which are included in $W$. Also, recall that a vertex $i$ of $G$ is a \emph{cut point} of $G$ if $G|_{[n]-\{i\}}$ has more connected components than $G$.  

Given $T\subset [n]$, a prime ideal $P_T(G)$ associated with $T$ is defined as 
\[
P_T(G)=(x_i,y_i: i\in T)+J_{\Tilde{G_1}}+\cdots+J_{\Tilde{G}_{c_G(T)}}
\]
where $G_1,\ldots,G_{c_G(T)}$ are the connected components of $G|_{[n]-T}$ and $\Tilde{H}$ is the complete graph on the vertex set $V(H)$. 

We say that $T\subset [n]$ has the \emph{cut point property} if each $i\in T$ is a cut point of the induced graph $G|_{([n]-T)\cup \{i\}}$ of $G$.
It was shown \cite{HHHKR} that the minimal prime ideals of $J_G$ are closely related to those subsets of the vertices of $G$ with the cut point property. More precisely,  
  \[ 
  \Min(J_G)=\{P_T(G): T=\emptyset~\textit{or}~T~\text{has the cut point property} \}. 
  \]
It is also known from \cite{HHHKR} that 
\begin{equation}\label{height formula}
    \height(P_T(G))=n+|T|-c_G(T)
\end{equation}
and 
\begin{equation}\label{dimension formula}
\dim(S/J_G)=\max \{n-|T|+c_G(T): T=\emptyset~\textit{or}~T~\text{has the cut point property}
\}.
\end{equation}

Recall that a \emph{chordal} graph is a graph whose induced cycles have length at most~$3$. Also, a \emph{block} graph is a chordal graph in which any two distinct maximal cliques (i.e. complete subgraphs) intersect in at most one vertex. 

\medskip
The main theorem of this section is the following.

\begin{Theorem}
    \label{main}
    Let $G$ be a Cameron--Walker graph on $n$ vertices. Then the following conditions are equivalent:
    \begin{itemize}
        \item[(i)] $S/J_G$ is Cohen--Macaulay;
        \item[(ii)] $J_G$ is unmixed; 
        \item[(iii)] $\dim (S/J_G) = n+1$; 
        \item[(iv)] Either $n = 3$ and $G$ is the path of length $2$ (Figure~\ref{P3}), or $n = 5$ and $G$ is the path of length $4$ (Figure~\ref{P5}), or $n=5$ and $G$ is obtained by attaching the path of length $2$ to a triangle (Figure~\ref{Block with a triangle}).
    \end{itemize}
\end{Theorem}

\begin{proof}
 Let $G$ be a Cameron--Walker graph whose connected bipartite graph has the vertex partition $U\cup V$. 
 
    (i)\implies (ii) is well--known.

    (ii)\implies (iii) follows from (\ref{dimension formula}), since $\height(P_{\emptyset}(G))=n-1$ by (\ref{height formula}). 

    (iii)\implies (iv) If there is a vertex $u\in U$ with at least two leaves, then for $T=\{u\}$, we have $n-|T|+c_G(T)\geq n+2$. This is impossible, because $\dim (S/J_G) = n+1$. Similarly, if there exists a vertex $v\in V$ with at least two triangles, then for $T=\{v\}$, we have $n-|T|+c_G(T)\geq n+2$ which is also impossible. Therefore, any vertex in $U$ has exactly one leaf and any vertex in $V$ has at most one triangle. Now, let $T=U$. Then $T$ has the cut point property, since $G$ is connected. Thus, $n-|T|+c_G(T)=n+|V|$, since in this case we have $c_G(T)=|U|+|V|$. Since $\dim (S/J_G) = n+1$, it follows that $|V|=1$. Let $V=\{v\}$. On the other hand, if $|U|\geq 3$, then by putting $T=\{v\}$, we get $n-|T|+c_G(T)\geq n+2$, a contradiction. Therefore, $|U|\leq 2$. Now, we distinguish two following cases:

    (1) Suppose that a triangle is attached to $v$. Then $|U|=1$, since otherwise the dimension is bigger than $n+1$ by considering $T=\{v\}$. Therefore,  $G$ is the graph obtained by attaching a path of length $2$ to a triangle (Figure~\ref{Block with a triangle}). 

    (2) Suppose that no triangle is attached to $v$. If $|U|=1$, then $G$ is the path of length $2$ (Figure~\ref{P3}). If $|U|=2$, then $G$ is the path of length $4$ (Figure~\ref{P5}).

    (iv)\implies (i) All of the three graphs described in (iv) are block graphs whose binomial edge ideals are Cohen--Macaulay (\cite[Theorem 1.1]{EHH}).
\end{proof}

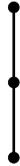
\begin{figure}[h!]
	\centering 
\begin{tikzpicture}[line cap=round,line join=round,>=triangle 45,x=1cm,y=1cm]
\draw [line width=1pt] (-6,5)-- (-6,3);
\begin{scriptsize}
\draw [fill=black] (-6,5) circle (2pt);
\draw [fill=black] (-6,4) circle (2pt);
\draw [fill=black] (-6,3) circle (2pt);
\end{scriptsize}
\end{tikzpicture} 
\caption{The path of length $2$}
\label{P3}
\end{figure}

\begin{figure}[h!]
	\centering 
\begin{tikzpicture}[line cap=round,line join=round,>=triangle 45,x=1cm,y=1cm]
\draw [line width=1pt] (-6,5)-- (-6,4);
\draw [line width=1pt] (-4,5)-- (-4,4);
\draw [line width=1pt] (-4,4)-- (-5,3);
\draw [line width=1pt] (-6,4)-- (-5,3);
\begin{scriptsize}
\draw [fill=black] (-6,5) circle (2pt);
\draw [fill=black] (-6,4) circle (2pt);
\draw [fill=black] (-5,3) circle (2pt);
\draw [fill=black] (-4,4) circle (2pt);
\draw [fill=black] (-4,5) circle (2pt);
\end{scriptsize}
\end{tikzpicture}
\caption{The path of length $4$}
\label{P5}
\end{figure}
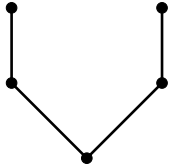

\begin{figure}[h!]
	\centering 
\begin{tikzpicture}[line cap=round,line join=round,>=triangle 45,x=1cm,y=1cm]
\draw [line width=1pt] (-6,5)-- (-6,4);
\draw [line width=1pt] (-6,4)-- (-6,3);
\draw [line width=1pt] (-6,3)-- (-7,2);
\draw [line width=1pt] (-6,3)-- (-5,2);
\draw [line width=1pt] (-5,2)-- (-7,2);
\begin{scriptsize}
\draw [fill=black] (-6,5) circle (2pt);
\draw [fill=black] (-6,4) circle (2pt);
\draw [fill=black] (-6,3) circle (2pt);
\draw [fill=black] (-7,2) circle (2pt);
\draw [fill=black] (-5,2) circle (2pt);
\end{scriptsize}
\end{tikzpicture}
\caption{The graph obtained by attaching the path of length~$2$ to a triangle}
\label{Block with a triangle}
\end{figure}
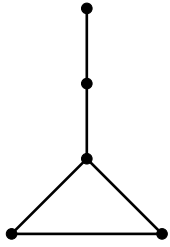

\section{Depth of binomial edge ideals of Cameron--Walker graphs}\label{depth section}

In this section we consider the depth of $S/J_G$ when $G$ is a Cameron--Walker graph and compute it in certain cases and make a comparison with the well--known lower and upper bounds given in~(\ref{bounds}).

Let $G$ be a finite connected graph on $n$ vertices. Recall that a vertex is called a \emph{free} vertex of $G$ if it is contained in exactly one maximal clique of $G$, and we denote the number of free vertices of $G$ by $f(G)$. 
Also, recall that the \emph{distance} $\dist_G(i,j)$ of two distinct vertices $i$ and $j$ in $[n]$ is the smallest length of paths connecting $i$ and $j$ in $G$. The \emph{diameter} of $G$ is
$d(G) = \max\{\dist_G(i,j):i,j \in[n]\}$. 
Moreover, we denote by $\kappa(G)$ the \emph{vertex connectivity} of $G$ which is the minimum number of vertices of $G$ whose deletion disconnects $G$. 

It is known \cite{BN, RSK2} that 
\[
f(G) + d(G) \leq \depth (S/J_G) \leq n - \kappa(G) + 2.
\]
If $G$ is a Cameron--Walker graph, then $\kappa(G) = 1$, and hence we have  
\begin{equation}\label{bounds}
   f(G) + d(G) \leq \depth (S/J_G) \leq n + 1.
\end{equation}


\medskip
We denote the bipartite subgraph of a Cameron--Walker graph $G$ induced on the vertex set $U\cup V$ by $B_G$. If $B_G$ is a tree, then $G$ is a block graph, and hence we have the next proposition which follows from \cite[Theorem~1.1]{EHH} where it was shown that $\depth (S/J_G)=n+1$ whenever $G$ is a connected block graph with $n$ vertices. Proposition~\ref{depth-tree} provides infinitely many Cameron--Walker graphs with nonCohen--Macaulay binomial edge ideals which have the maximum depth according to (\ref{bounds}). 

\begin{Proposition}\label{depth-tree}
    Let $G$ be a Cameron--Walker graph with $n$ vertices such that $B_G$ is a tree. Then 
    \[
    \depth(S/J_G)=n+1.
    \]
\end{Proposition}

Next, it is natural to ask about smaller values of depth and in particular attaining the lower bound given in (\ref{bounds}) for binomial edge ideals of Cameron--Walker graphs. We start with the following example.

\begin{Example}\label{smallest example}
    Let $G$ be the Cameron--Walker graph on $6$ vertices shown in Figure~\ref{induction base}. This graph is indeed the smallest Cameron--Walker graph $H$ for which $B_{H}$ is not a tree. Computations by \emph{CoCoA} \cite{CoCoA} show that $\depth(S/J_{H})=6$, namely equal to the number of vertices of $H$.   
\end{Example}

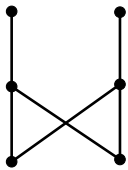
\begin{figure}[h!]
	\centering
\begin{tikzpicture}[line cap=round,line join=round,>=triangle 45,x=1cm,y=1cm]
\draw [line width=1pt] (-5,5)-- (-5,4);
\draw [line width=1pt] (-5,4)-- (-5,3);
\draw [line width=1pt] (-3.56,4.99)-- (-3.56,4.03);
\draw [line width=1pt] (-3.56,4.03)-- (-3.56,3.03);
\draw [line width=1pt] (-5,4)-- (-3.56,3.03);
\draw [line width=1pt] (-5,3)-- (-3.56,4.03);
\begin{scriptsize}
\draw [fill=black] (-5,4) circle (2pt);
\draw [fill=black] (-5,3) circle (2pt);
\draw [fill=black] (-3.56,4.03) circle (2pt);
\draw [fill=black] (-3.56,3.03) circle (2pt);
\draw [fill=black] (-5,5) circle (2pt);
\draw [fill=black] (-3.56,4.99) circle (2pt);
\end{scriptsize}
\end{tikzpicture}
\caption{The smallest Cameron--Walker graph $H$ for which $B_H$ is a not a tree}
\label{induction base}
\end{figure}

Motivated by the finite graph $H$ shown in Figure~\ref{induction base}, in the following we consider a class of Cameron--Walker graphs in which $B_G$ is not a tree. First, note that we denote by $K_t$ the complete graph on $t$ vertices. A graph is called \emph{triangle--free} if it does not have any induced subgraph isomorphic to $K_3$.  

Now, let $G$ be a Cameron--Walker graph such that $B_G$ is a complete bipartite graph which is not a tree. Then we call $G$ a \emph{dense Cameron--Walker} graph. If $G$ is triangle--free and each vertex in $U$ has only one attached leaf, then we say that $G$ is a \emph{minimal dense Cameron--Walker} graph. The graph $H$ depicted in Figure~\ref{induction base} is indeed the smallest minimal dense Cameron--Walker graph.    

In the sequel, we give an explicit formula for the depth of $S/J_G$ where $G$ is a minimal dense Cameron--Walker graph. For this purpose, we need to recall some further ingredients. The following lemma, as a part of the well--known depth lemma, is helpful to prove our theorem. 

\begin{Lemma}\label{depth lemma}
 Let $R$ be a standard graded polynomial ring over the field $K$, and let $M$, $N$ and $P$ be finitely generated graded $R$-modules. If 
 \[
 0\rightarrow M\rightarrow N\rightarrow P\rightarrow 0
 \]
is a short exact sequence and if $\depth(N) > \depth(P)$, then
\[
\depth(M)=\depth(P)+1.
\]
\end{Lemma}

\medskip
In \cite{KS} the class of \emph{generalized block graphs} was defined as a generalization of the class of block graphs. Let $G$ be a chordal graph such that for every three maximal cliques of $G$ whose sets of vertices have a nonempty intersection, the intersection of the vertex sets of each pair of them is the same, (see for example~\cite[Figure~1 and Figure~2]{KS}). It is clear that every block graph is also a generalized block graph. An explicit formula for $\depth(S/J_G)$ is known for any generalized block graph $G$. To state that formula, we need some more concepts to recall. By a \emph{cut set} of a finite graph $G$, we mean a subset of vertices of $G$ whose deletion results in a graph which has more connected components than $G$. Moreover, by a \emph{minimal cut set} of $G$, we mean a cut set which is minimal under inclusion. Also, recall that the \emph{clique number} of $G$, denoted by $\omega(G)$, is the maximum size of the maximal cliques of $G$. Now, let $G$ be a generalized block graph on $[n]$. For any 
$i = 1,\ldots,\omega(G)-1$, let 
\[
\mathcal{A}_i(G) = \{A \subseteq [n] : |A| = i, A~\textit{is~a~minimal~cut~set~of}~G\}.
\]
The next theorem gives the aforementioned formula in terms of $|\mathcal{A}_i(G)|$. 

\begin{Theorem}\label{generalized block}
 \cite[Theorem~3.2]{KS}
 Let $G$ be a connected generalized block graph on $[n]$. Then
 \[
 \depth(S/J_G) = n+1-\sum_{i=2}^{\omega(G)-1} (i-1)|\mathcal{A}_i(G)|.
 \] 
\end{Theorem}

 In the following, we sometimes need to deal with different subgraphs of a given finite graph simultaneously, and consequently with different polynomial rings over $K$. As usual in the literature, we denote by $S_w$ the polynomial ring over $K$ on the same variables of $S$ except $x_w$ and $y_w$ where $w$ is a vertex of $G$ which is removed from $G$. But in other cases where we remove more vertices from the underlying graph, to avoid fixing complicated notation, we simply emphasize in words that we consider \emph{the desired polynomial ring over $K$ with variables corresponding to the vertices of the underlying graph}, which means that we are removing the two corresponding variables to any of those omitted vertices. 

 \medskip
 Applying the Auslander--Buchsbaum formula to \cite[Proposition~1.3]{HR} implies the following

\begin{Proposition}\label{gluing graphs}
(see \cite[Proposition~1.3]{HR} )
    Let $G_1$ and $G_2$ be two graphs on disjoint sets of vertices, and let $G$ be the graph obtained by attaching $G_1$ and $G_2$ by identifying a vertex which is a free vertex in both $G_1$ and $G_2$. Then
    \[
    \depth(S/J_G)=\depth(S_1/J_{G_1})+\depth(S_2/J_{G_2})-2,
    \]
    where $S$, $S_1$ and $S_2$ are the desired polynomial rings over the field $K$ corresponding to the vertex set of $G$, $G_1$ and $G_2$, respectively.
\end{Proposition}

\medskip
In the next theorem we consider the class of minimal dense Cameron--Walker graphs and compute the depth of their binomial edge ideal which in particular never attain the upper bound in (\ref{bounds}). First we fix some notation. 

Let $G$ be a finite graph on $[n]$ and $w$ a vertex of $G$. Then, denoted by $G-w$ we mean the induced subgraph $G|_{[n]-\{w\}}$ of $G$. Moreover, denoted by $G_w$ we mean the graph on $[n]$ whose edge set is 
\[
E(G)\cup \big{\{}\{a,b\}: \{a,b\}\subseteq N_G(w)\big{\}},
\]
where $N_G(w)$ is the set of all adjacent vertices to $w$ in $G$. We may also denote the vertex set of a graph $G$ by $V(G)$.  



\medskip
Now, we are ready to prove the main theorem of this section.
    
\begin{Theorem}\label{depth-nontree}
    Let $G$ be a minimal dense Cameron--Walker graph with $n$ vertices and let $U\cup V$ be the vertex partition for $B_G$. 
    Then 
    \[
    \depth(S/J_G)=n-|V|+2.
    \]
\end{Theorem}

\begin{proof}
Since $B_G$ is not a tree, we have $n\geq 6$ and $|U|, |V|\geq 2$. Let 
$U=\{u_1,\ldots,u_r\}$ with $r\geq 2$.
Let $u'_i$ be the only leaf attached to $u_i$ for each $i=1,\ldots,r$. We prove the theorem using induction on the number of vertices $n$. 
For $n=6$ the only possible case was observed in Example~\ref{smallest example} for which the result holds. Now, assume that $n>6$. 

By \cite[Lemma~4.8]{O}, $J_G=J_{G_{u_1}}\cap ((x_{u_1},y_{u_1})+J_{G-u_1})$. Therefore, we have the exact sequence  
\[
0 \rightarrow S/J_G \rightarrow S/J_{G_{u_1}}\oplus S_{u_1}/J_{G-u_1} \rightarrow S_{u_1}/J_{G_{u_1}-u_1} \rightarrow 0.
\]

Let $G'$ be the finite graph consisting of $r$ maximal cliques $F_1,\ldots,F_r$ where $V(F_1)=\{u_1,u'_1\}\cup V$ and $V(F_i)=\{u_i\}\cup V$ for each $i=2,\ldots,r$. Then, $G'$ is a generalized block graph with $n-|U|+1$ vertices and with exactly one minimal cut set which has cardinality~$|V|$. Therefore, it follows from Theorem~\ref{generalized block} that 
\[
\depth(S'/J_{G'})=n-|U|-|V|+3,
\]
where $S'$ is the desired polynomial ring over $K$ with variables corresponding to the verices of $G'$. Note that $G_{u_1}$ is the graph obtained by attaching the leaves $u'_2,\ldots,u'_r$ to the free vertices $u_2,\ldots,u_r$ of $G'$, respectively. Then, it follows from Proposition~\ref{gluing graphs} that 
\begin{equation}\label{A}
\depth(S/J_{G_{u_1}})=n-|U|-|V|+3+|U|-1=n-|V|+2,    
\end{equation}
since $\depth(R/J_{K_2})=3$ where $R$ is the desired polynomial ring on $4$ variables corresponding to the two vertices of $K_2$. 

Let $G''=G'-u_1$. Then $G''$ is a generalized block graph consisting of the maximal cliques $F_1-u_1,F_2,\ldots,F_r$ with $n-|U|$ vertices and with exactly one minimal cut set which has cardinality~$|V|$. Therefore, it follows from Theorem~\ref{generalized block} that 
\[
\depth(S''/J_{G''})=n-|U|-|V|+2,
\] 
where $S''$ is the desired polynomial ring over $K$ with variables corresponding to the verices of $G''$. Note that $G_{u_1}-u_1$ is the graph obtained by attaching the leaves $u'_2,\ldots,u'_r$ to the free vertices $u_2,\ldots,u_r$ of $G''$, respectively. Then, it follows from Proposition~\ref{gluing graphs} that 
\begin{equation}\label{B}
\depth(S_{u_1}/J_{G_{u_1}-u_1})=n-|U|-|V|+2+|U|-1=n-|V|+1.
\end{equation}

The finite graph $G-u_1$ is the disjoint union of an isolated vertex $u'_1$ and the triangle--free Cameron--Walker graph $G|_{W}$ where $W=[n]-\{u_1,u'_1\}$ such that $B_{G|_{W}}$ is a complete bipartite graph on $(U-\{u_1\})\cup V$ and each vertex in $U-\{u_1\}$ has exactly one attached leaf. Now, we distinguish two cases:

(i) If $|U|=2$, then $B_{G|_{W}}$ is a tree. Thus, it follows from Proposition~\ref{depth-tree} that
\[
\depth(S'''/J_{G|_{W}})=n-1,
\]
where $S'''$ is the desired polynomial ring over $K$ with variables corresponding to the vertices of $G|_{W}$, and hence 
\begin{equation}\label{C}
   \depth(S_{u_1}/J_{G-u_1})=n+1, 
\end{equation}
since $u'_1$ is just an isolated vertex in $G-u_1$. 

(ii) If $|U|\geq 3$, then $B_{G|_{W}}$ is not a tree. Thus, $G-u_1$ is the disjoint union of an isolated vertex $u'_1$ and a minimal dense Cameron--Walker graph $G|_{W}$. Then, the induction hypothesis yields that 
\[
\depth(S'''/J_{G|_{W}})=(n-2)-|V|+2=n-|V|,
\]
and hence 
\begin{equation}\label{C'}
   \depth(S_{u_1}/J_{G-u_1})=n-|V|+2, 
\end{equation}
since $u'_1$ is just an isolated vertex in $G-u_1$. 

Note that (\ref{A}) and (\ref{C}) as well as (\ref{A}) and (\ref{C'}) imply that 
\[
\depth(S/J_{G_{u_1}}\oplus S_{u_1}/J_{G-u_1})=n-|V|+2,
\]
since $|V|\geq 2$. This together with (\ref{B}) and Lemma~\ref{depth lemma}, imply that
\[
\depth(S/J_G)=n-|V|+2,
\]
as desired. 
\end{proof}


\begin{Remark}\label{other formula}
   Let $G$ be a minimal dense Cameron--Walker graph. Then, $n=2|U|+|V|$, and hence it is clear from the formula given in Theorem~\ref{depth-nontree} that 
   \[
   \depth(S/J_G)=2|U|+2.
   \]
   We chose the form $n-|V|+2$ in the statement of the theorem just because it is more compatible with the upper bound for depth in (\ref{bounds}).  
\end{Remark}

It is clear from the formula given in Theorem~\ref{depth-nontree} that depth can never attain the upper bound~$n+1$ for any minimal dense Cameron--Walker graph $G$. But, this value can meet the lower bound $f(G)+d(G)$ as we show in the next corollary.

\begin{Corollary}\label{equal with the lower bound}
    Let $G$ be a minimal dense Cameron--Walker graph and let $U\cup V$ be the vertex partition for $B_G$. Then the following conditions are equivalent:
    \begin{enumerate}
        \item $\depth (S/J_G)=f(G)+d(G)$;
        \item $\depth(S/J_G)=6$;
        \item $|U|=2$.
    \end{enumerate}
\end{Corollary}

\begin{proof}
Note that since $|U|\geq 2$, the only free vertices of $G$ are the leaves attached to the vertices in $U$, and hence $f(G)=|U|$. 
Moreover, the diameter of $G$ is attained by the distance of two distinct leaves and so that $d(G)=4$.  
Thus, $f(G)+d(G)=|U|+4$. Therefore, by Remark~\ref{other formula}, we need to consider the equality $|U|+4=2|U|+2$ which is clearly equivalent to $|U|=2$ as well as $\depth(S/J_G)=6$. Thus, the result follows.     
\end{proof}

The finite graphs $G$ for which $\depth(S/J_G)\leq 5$ have been characterized \cite{RSK1}, \cite{RSK2}. The next interesting characterization would be for the finite graphs $G$ with $\depth(S/J_G)=6$. Just as a step towards this goal, Corollary~\ref{equal with the lower bound}, in particular, provided an infinitely many graphs with this property. More precisely, we have

\begin{Corollary}\label{any n with depth 6}
    For any integer $n\geq 6$, there exists a minimal dense Cameron--Walker graph $G$ with $n$ vertices and $\depth(S/J_G)=6$. 
\end{Corollary}

\begin{proof}
 Since $n\geq 6$, we set $G$ to be the minimal dense Cameron--Walker graph with $|U|=2$ and $|V|=n-4$. Then, it follows from Corollary~\ref{equal with the lower bound} that $\depth(S/J_G)=6$.  
\end{proof}


\section{Possible values for depth of any binomial edge ideal}\label{possible depth section}

We now consider all the possible values for depth of binomial edge ideals. Indeed, motivated by Corollary~\ref{any n with depth 6} together with \cite{RSK1}, \cite{RSK2}, a natural question is the following

\begin{Question}
    For given integers $t$ and $n$ with $6\leq t\leq n+1$, does there exist a finite connected graph $G$ with $n$ vertices such that $\depth(S/J_G)=t$?
\end{Question}

As an application of our results in Section~\ref{depth section}, we positively answer this question in the next theorem. This question was also discussed in \cite[Corollary~3.2]{FS}. 

\begin{Theorem}\label{possible values of depth}
    Given integers $t$ and $n$ such that $6\leq t\leq n+1$, there exists a finite connected graph $G$ with $n$ vertices and $\depth(S/J_G)=t$.
\end{Theorem}

\begin{proof}
First, suppose that $t=n+1$. Then, by Proposition~\ref{depth-tree}, for any Cameron--Walker graph $G$ with $n$ vertices whose $B_G$ is a tree, we have $\depth(S/J_G)=t$. So, we may consider $G$ to be the star graph with $n$ vertices which can also be seen as the Cameron--Walker graph with $|U|=1$ and $|V|=n-2$. 

Next, suppose that $t\leq n$. Then, we consider the following two cases:    

(i) Assume that $t$ is even. Then, let $U$ be a set of vertices with $|U|=(t-2)/2$, and let $V$ be a set of vertices disjoint from $U$ with $|V|=n-t+2$. Note that $|U|, |V|\geq 2$. Then, let $G$ be the minimal dense Cameron--Walker graph for which the partition of the vertices of $B_G$ is $U\cup V$. Thus, it follows from Remark~\ref{other formula} that $\depth(S/J_G)=t$.

(ii) Assume that $t$ is odd. Then, let $U$ be a set of vertices with $|U|=(t-3)/2$, and let $V$ be a set of vertices disjoint from $U$ with $|V|=n-t+2$. Note that $|U|, |V|\geq 2$. Now, let $G'$ be the minimal dense Cameron--Walker graph for which the partition of the vertices of $B_G$ is $U\cup V$. By Remark~\ref{other formula}, we have $\depth(S/J_{G'})=t-1$. Now, let $G$ be the finite graph obtained by attaching one new leaf to one of the leaves of $G'$. Then, Proposition~\ref{gluing graphs} implies that $\depth(S/J_G)=\depth(S/J_{G'})+1=t$, since $\depth(R/J_{K_2})=3$ where $R$ is a desired polynomial ring over $K$ with four variables corresponding to the two vertices of $K_2$. This completes the proof. 
\end{proof}

We conclude the present paper with

\begin{Question}
    Given integers $t,n,d$ with $4 \leq t \leq n+1 \leq d \leq 2n-2$, does there exist a finite connected graph $G$ on $n$ vertices for which $\dim(S/J_G) = d$ and $\depth(S/J_G) = t$?
\end{Question}

\section*{Acknowledgment}

The present paper was completed while the authors stayed at Sabanc\i~\"Universitesi, Istanbul, Turkey, August~05 to~14, 2025.  Sara Saeedi Madani was in part supported by a grant from IPM (No. 1404130019).

\end{document}